\documentclass[leqno,11pt]{scrartcl}
\usepackage[utf8]{inputenc}

\usepackage{amsmath,amssymb,amsthm}
\usepackage{amscd}
\usepackage{setspace} %%% Einstellung des Zeilenvorschubs mit \spacing{ } %%%
\usepackage{enumerate}
\usepackage{array}
\usepackage{tabularx}
\usepackage{multirow}
\usepackage{booktabs}

\theoremstyle{definition}
\newtheorem{definition}{Definition}

\newtheorem{remark}{Remark}

\theoremstyle{plain}
\newtheorem{theorem}{Theorem}

\newtheorem{lemma}[definition]{Lemma}
\newtheorem{corollary}{Corollary}
\newtheorem{proposition}{Proposition}
\theoremstyle{remark}

\newcommand{\N}{\mathbb{N}}
\newcommand{\Z}{\mathbb{Z}}
\newcommand{\Q}{\mathbb{Q}}
\newcommand{\R}{\mathbb{R}}
\newcommand{\C}{\mathbb{C}}

\newcommand{\h}{\mathcal{H}} %orthogonaler Halbraum
\newcommand{\M}{\mathcal{M}} %Modulformen
\renewcommand{\P}{\mathbb{P}}

\newcommand{\Ccal}{\mathcal{C}}

\newcommand{\disc}{\operatorname{disc}}

\newcommand{\im}{\operatorname{Im}}
\newcommand{\GL}{\operatorname{GL}}%generelle lineare Gruppe
\newcommand{\Pos}{\mathcal{P}}%positiv definite Matrizen

\newcommand{\gsi}{\Gamma_{S,\infty}}

\newcommand{\T}{^{tr}}%transpose

\newcommand{\EkEll}{E_{k,\text{ell}}}

\newcommand{\EkSing}{E_{k,S}^{(0)}}
\newcommand{\EkSingC}{E_{k,S,c}^{(0)}}

\newcommand{\EkCusp}{E_{k,S}^{\bullet}}
\newcommand{\EkCuspC}{E_{k,S,c}^{\bullet}}
\newcommand{\ConstCusp}{c_{k,S}^{\bullet}}

\newcommand{\GammaD}{\widetilde{\Gamma}_S}

\renewcommand{\mod}{\, \operatorname{mod}}

\newcommand{\dual}{^{\#}}
\let\leq\leqslant
\let\geq\geqslant

\begin{document}

\begin{center}
\begin{huge}
\begin{spacing}{1.0}
\textbf{Fourier coefficients of Eisenstein series for $O^+(2,n+2)$}  
\end{spacing}
\end{huge}

\bigskip
by
\bigskip

\begin{large}
\textbf{Felix Schaps\footnote{Felix Schaps, Lehrstuhl A für Mathematik, RWTH Aachen University, D-52056 Aachen, Germany, felix.schaps@matha.rwth-aachen.de}}
\end{large}
\vspace{0.5cm}\\
Datum
\vspace{1cm}
\end{center}
\begin{abstract}
We derive the Fourier expansion of (scalar-valued) Eisenstein series for $O(2,n+2)$ using classical methods of Siegel, Braun, Zagier, Bruinier and others. We assume that the underlying lattice splits two hyperbolic planes.
Finally we prove for the Eisenstein series at the standard cusp 
that they belong to the Maaß space for $O(2,n+2)$, an analogue of the 'Spezialschar' for Siegel modular forms introduced by Maaß, at least at all local places $p$, where the localization of the underlying lattice is maximal.
\end{abstract}
\noindent\textbf{Classification: 11F30, 11F55}
\vspace{2ex}\\

\newpage
\section{Introduction}

The classical elliptic Eisenstein series
\[
 \EkEll(\tau) = \frac{1}{2\zeta(k)} \cdot \sum_{(\gamma,\delta) \in \Z \times \Z \backslash \{ (0,0)\}} (\gamma \tau + \delta)^{-k}, \, \text{Im}(\tau)>0, k>2,
\]
possesses the Fourier expansion
\[
 \EkEll(\tau) = 1 +  - \frac{2k}{B_k} \cdot \sum_{m=1}^\infty \sigma_{k-1}(m) e^{2\pi i m \tau},
\]
where $B_k$ is the $k$-th Bernoulli number and $\sigma_k(m)$ the weighted divisor sum.

In this paper we examine the Fourier expansion of Eisenstein series for orthogonal groups of signature $(2,n+2)$ splitting two hyperbolic planes.
The Hermitian symmetric space associated with the orthogonal group $O(2,n+2)$ is a Siegel domain of type IV. The attached spaces of modular forms have attracted attention due to Borcherds (cf. \cite{Bo2}).
If $n=1$, the space of modular forms is isomorphic to Siegel modular forms of degree $2$ or paramodular forms, respectively (cf. \cite{K1}); if $n=2$, in some cases the orthogonal groups are isomorphic to a subgroup of the Hermitian modular group of degree $2$ for an imaginary quadratic field, cf. \cite{We}.
The Fourier expansion of Siegel Eisenstein series was first examined in \cite{Si} and Maaß \cite{Maass_Fourier1} proved that the Fourier coefficients satisfy a certain recursion formula. Such modular forms belong to the Maaß space.
A generalization of the Maaß space for $O(2,n+2)$ can be found in \cite{K8}, the Maaß lift or additive lift has been described by Gritsenko (cf. \cite{G4}).

The Fourier expansion of the Eisenstein series for the standard cusp and maximal lattices have already been investigated in \cite{Hirai_Eisenstein} using adeles.
Here, we give a different approach which is similar to the standard approach by Siegel \cite{Si}, also used in \cite{B} and to Bruinier's work \cite{Bruinier_Eisenstein}. 
With this method we are able to examine the Fourier coefficients even for non-maximal lattices. It turns out that the Eisenstein series belong to the Maaß space at every place where the localization is maximal.
It follows that the Eisenstein series considered here are the Maaß lift of Jacobi-Eisenstein series considered in \cite{Mocanu_PhD}.

\section{Modular forms for $O(2,n+2)$}

We start with an even lattice $\Lambda$ in a $\Q$-vector space $V$ of dimension $n$ equipped with a non-degenerate symmetric bilinear form $( \cdot,\cdot )$, i.e. $\Lambda$ is a free group of rank $n$ satisfying $( \lambda,\lambda ) \in 2\Z$ for all $\lambda \in \Lambda$. The dual lattice is given by 
\[
 \Lambda^\sharp:= \{v\in V;\; ( v,\Lambda ) \subseteq \Z\} \supseteq \Lambda.
\]
The associated quadratic form is $Q(v)=\frac{1}{2}(v,v)$.
Throughout the paper we choose a basis of a positive definite lattice $L$ with Gram matrix $S$. Let $\disc L:= \det S$ denote its discriminant. We add two hyperbolic planes over $\Z$, i.e.
\begin{gather*}\tag{1}\label{gl_1}
 \begin{cases}
  & L=\Z^n,\quad S\in\Z^{n\times n} \:\text{positive definite and even}, \\[1ex]
  & L_0 = \Z^{n+2}, \quad S_0 = \left(\begin{smallmatrix}
                                 0 & 0 & 1 \\ 0 & -S & 0 \\ 1 & 0 & 0
                                \end{smallmatrix}\right),   \\[2ex]
  & L_1 = \Z^{n+4},\quad S_1 = \left(\begin{smallmatrix}
                                 0 & 0 & 1 \\ 0 & S_0 & 0 \\ 1 & 0 & 0
                                \end{smallmatrix}\right). 
 \end{cases}
\end{gather*}
Thus $S_1$ has got the signature $(2,n+2)$.

Superscripts $0$ and $1$ always mean that we consider the considered term for the indefinite lattice $L_0$ or $L_1$, respectively, instead of $L$.

We consider the attached real orthogonal group
\begin{align*}
O(S_1;\R)&=\{M\in \R^{(n+4) \times (n+4)};\;M\T S_1 M=S_1\}\\
& =\{M \in \GL(n+4,\R); \, Q_1(Mg)=Q_1(g) \text{ for all } g\in \R^{n+4}\}.
\end{align*}
Let $O^+(S_1;\R)$ stand for the connected component of the identity matrix $I$. Given $M\in O^+(S_1;\R)$ we will always assume the form 
\begin{gather}\label{gl_2}
 M = \begin{pmatrix}
      \alpha & a^{tr} S_0 & \beta \\ b & K & c \\ \gamma & d^{tr} S_0 & \delta
     \end{pmatrix},\; \alpha, \beta, \gamma, \delta \in\R, \;\; a,b,c,d\in\R^{n+2}, \; K\in \R^{(n+2)\times(n+2)}.
\end{gather}
Its inverse is given by
\begin{gather}\label{gl_3}
 M^{-1} = S_1^{-1} M^{tr} S_1 = \begin{pmatrix}
                                 \delta & c^{tr}S_0 & \beta \\ d & S_0^{-1} K^{tr} S_0 & a \\
                                 \gamma & b^{tr} S_0 & \alpha
                                \end{pmatrix}.
\end{gather}
Let $\Gamma_S:=O^+(S_1;\Z)$ denote the subgroup of integral matrices. Note that in this case $a,d\in \Z^{n+2}$ holds in \eqref{gl_2} due to $M^{-1}\in\Gamma_S$ and \eqref{gl_3}.
We define the \emph{discriminant kernel} 
\[
 \widetilde{\Gamma}_S := \{M\in \Gamma_S;\; M\in I + \Z^{(n+4)\times (n+4)} S_1\},
\]
where $I$ is the identity matrix. The discriminant kernel induces the identity on $L^\sharp_1/L_1$, $L^\sharp_1 = S^{-1}_1 \Z^{n+4}$.

At first we adopt a description of the first columns of matrices in $\Gamma_S$ from \cite{KS22}:
\begin{theorem}\label{theorem_1} %%% Theorem 1
 Let $L_1=\Z^{n+4}$ satisfy \eqref{gl_1}. Given $h\in\Z^{n+4}$ the following assertions are equivalent:
 \begin{enumerate}
  \item[(i)] $h$ is the first column of a matrix in $\Gamma_S$ (resp. $\widetilde{\Gamma}_S$).
  \item[(ii)] $h^{tr}S_1$ is the last row of a matrix in $\Gamma_S$ (resp. $\widetilde{\Gamma}_S$).
  \item[(iii)] $Q_1(h)=0$ and $\gcd(S_1 h)=1$.
 \end{enumerate}
\end{theorem}

We consider particular matrices in $\widetilde{\Gamma}_S$:
 \begin{gather}\label{gl_4}
  J= \begin{pmatrix}
     0 &  0 & 0 & 0 & -1 \\ 0 & 0 & 0 & -1 & 0 \\
      0 & 0 & I & 0 & 0 \\ 
      0 & -1 & 0 & 0 & 0 \\
      -1 & 0 & 0 & 0 & 0 \\
     \end{pmatrix}, \quad
  R_K = \begin{pmatrix}
      1 & 0 & 0 \\ 0 & K & 0\\ 0 & 0 & 1
     \end{pmatrix}, R\in O^+(S_0;\Z),
 \end{gather}
\begin{gather}\label{gl_5}
 T_\lambda = \begin{pmatrix}
              1 & -\lambda^{tr} S_0 & -\tfrac{1}{2} \lambda^{tr} S_0 \lambda \\ 0 & I & \lambda \\ 0 & 0 & 1 
             \end{pmatrix}, \quad
 \widetilde{T}_\lambda = \begin{pmatrix}
               1 & 0 & 0 \\ \lambda & I & 0 \\ -\tfrac{1}{2} \lambda^{tr} S_0 \lambda & -\lambda^{tr} S_0 & 1
              \end{pmatrix},\;\;
 \lambda \in \Z^{n+2}.
\end{gather}

Define the positive cone $\Pos_S:=\{ y \in \R^{n+2}; \; Q_0(y)>0, y_1 > 0 \}.$
$O(S_1;\R)$ acts on $\h_S := \left\{ z = x +iy \in \C^{n+2}; \; y \in \Pos_S \right\}$, a Hermitian domain of type IV (cf. \cite{G4}), via
\begin{align}\label{eq:action}
M \langle z \rangle &:= \left(- Q_0(z) b + Kz + c\right) \cdot (j(M,z))^{-1}, \\
\text{ and } j(M,z) &:= - \gamma Q_0(z) + (d,z)_0 + \delta,
\end{align}
if $M$ is given in the decomposition as in (\ref{gl_2}).

There is an isomorphic projective model for the domain, cf. \cite{Bru2}. We consider cusps in this context.
If $\Gamma$ is a subgroup of $\Gamma_S$ of finite index, we denote by
\[
 \Ccal^0(\Gamma):=\bigl\{\pm \Gamma h;\; h\in L^\sharp_1 = S^{-1}_1 \Z^{n+4},\,h^{tr} S_1 h=0,\, \gcd(S_1 h) = 1\bigr\}
\]
the set of $\Gamma$-orbits of \emph{zero-dimensional cusps}.

Moreover we define a subgroup $\gsi$ by
\[ \gsi := \{ M \in \Gamma_S; \; \gamma = 0, \, d=0,\, \delta = 1\},\]
which turns out to be the stabilizer of the standard cusp $e_1=(1,0,\cdots,0)\T$.
Due to Proposition 3 in \cite{K1}, one has the following description:
\begin{proposition}\label{prop:gsi}
Every $M \in \gsi$ has a unique decomposition $M=T_c R_K = R_K T_{K^{-1}c}$ where $K \in O^{+}(S_0;\Z)$ and $c \in \Z^{n+2}$.
In particular, $M$ is of the form
\[
M=\begin{pmatrix}
1 & -c\T S_0 & -Q_0(c) \\ 0 & K & c \\ 0 & 0 & 1
\end{pmatrix}.
\]
\end{proposition}

The Petersson slash operator for a function $f:\h_S \to \C$ is defined by
\[
f|_k M (z)= j(M,z)^{-k} \cdot f(M\langle z \rangle) \text{ for all } z\in \h_S,
\]
where $M \in O(S_1;\R)$.

We call a holomorphic function $f:\h_S\to\C$ a \emph{modular form} of weight $k$ for a subgroup $\Gamma$ of finite index in $\Gamma_S$, if $f|_k M=f$ for all $M\in \Gamma$.
Such modular forms possess a Fourier expansion of the form
$$f(z) = \sum_{\lambda \in L_0\dual \cap \overline{\Pos_S}} a_f(\lambda) e^{2\pi i (\lambda,z)_0},$$
where $\overline{\Pos_S}$ denotes the topological closure and $\partial \Pos_S$ the boundary.

%%%%%%%%%%%%%%%%%%%%%%%%%%%%%%%%%%%%%%%%%%%%%%%%%%%%%%%%%%%%%%%%%%%%%%%%%%%%%%%%%%%%
%%%%%%%%%%%%%%%%%%%%%%%%%%%%%%%%%%%%%%%%%%%%%%%%%%%%%%%%%%%%%%%%%%%%%%%%%%%%%%%%%%%%
\section{Maximal lattices}
Maximal (even) lattices play an important role while examing orthogonal modular forms, cf. \cite{KS22}.
Let $R$ be a principal ideal domain. An even $R$-lattice $L$ is called \emph{maximal} if for any even lattice $M$ satisfying $L \subset M$ (and the rank of both lattices coincides) we already have $L = M$.

Consider the localizations $L_p=\Z_p^n$ of the lattice $L$ over the quotient field $\Q_p$. We embed the rational number $\Q$ into the local field $\Q_p$ in the natural way. A well-known result (cf. \cite{Voight}) is
\begin{proposition}\label{prop:local}
 The lattice $L$ is maximal if and only if all its localizations $L_p$ are maximal.
\end{proposition}
In the present setting the maximality for the localization can be characterised, cf. \cite[Prop. 1.6.10]{Schaps_Diss}.
\begin{proposition}\label{prop:local_appl}
The local lattice $L_p$ is maximal if and only if
\[
Q(v) \in \Z, \, v \in L\dual
\]
implies that $p \nmid \operatorname{ord}_{L\dual/L}(v)$.
\end{proposition}

The order of an element of the discriminant group is a divisor of $\det(S)$. Hence, the localizations $L_p$ are all maximal except of finitely many ones. We formulate it precisely:
\begin{corollary}
Let $p\nmid \det(S)$, then $L_p$ is maximal.
\end{corollary}

\section{The Maaß space}
Maaß \cite{Maass_Fourier} observed that the Fourier coefficients of Siegel Eisenstein series of degree $2$ satisfy
an arithmetic recursion. Functions satisfying such a relation belong to the so-called Maaß space. Maaß originally called it the "Spezialschar".
In \cite{K8} there is a generalization to modular forms for the orthogonal group.

From here on, we identify $L_0\dual \cong \Z \times S^{-1}\Z^n \times \Z$ in the context of Fourier coefficients of vectors.
\begin{definition}\label{def:Maass_space}
Let $\Gamma$ be a subgroup of $\Gamma_S$ of finite index. The \emph{Maaß space} $\M_k^{*}(\Gamma)$ consists of all $f\in \M_k(\Gamma)$ satisfying
\[
a_f(\lambda) = \sum_{d|\gcd(S_0 \lambda)} d^{k-1} a_f\left(\frac{lm}{d^2},\frac{\mu}{d},1\right) \text{ for all } \lambda= (l,\mu,m) \in L_0\dual \cap \overline{\Pos_S} \backslash \{ 0 \}.
\]
\end{definition}

There is a local-global principle for maximal lattices, i.e. a lattice is maximal if and only if it is maximal for all localizations (see Proposition \ref{prop:local}). A similar local-global principle is true for the Maaß space.
\begin{lemma}
 A modular form $f\in \M_k(\Gamma)$ belongs to $\M_k^{*}(\Gamma)$ if and only if for all primes $p \in \P$ the following identity holds
\begin{align*}\label{eq:local_Maass}
a_{f}(\lambda) = a_{f}(p^t l,\mu,m) + p^{k-1} a_{f}\left(\frac{1}{p} \lambda\right), \text{ for all $\lambda=(l,\mu,p^t m)\in L_0\otimes_\Z \Q$, $t \in \N$, $p\nmid m$.}
\end{align*}
Here $a_f(\lambda)$ is extended to $\lambda\in L_0\otimes_\Z \Q$ by setting $a_f(\lambda)=0$, if the value was not already defined.
\end{lemma}
This is called the \emph{local Maaß condition}.
\begin{proof}
Noting that
\begin{align*}
\sum_{d|\gcd(S_0 \lambda)} d^{k-1} a_f\left(\frac{lm}{d^2},\frac{\mu}{d},1\right)
=\sum_{d|m} d^{k-1} a_f\left(\frac{lm}{d^2}, \frac{\mu}{d}, 1\right),
\end{align*}
gives "$\Rightarrow$". For "$\Leftarrow$", use the same identity. The claim follows by induction on $N:=\sigma_0(m)$.
\end{proof}

\section{Eisenstein series for different cusps}

For every $0$-dimensional cusp we can attach an Eisenstein series.
The procedure is comparable with \cite{Elstrodt_hyperbolic}.
Let $\GammaD \subset \Gamma \subset \Gamma_S$ and $c$ be a $0$-dimensional cusp for $\Gamma$. Then we define the stabilizer of the cusp $c$ by
\[
\Gamma_c := \{ M \in \Gamma; \, M c = c\}.
\]
The standard cusp is $e_1:=(1,0,\cdots,0)\T$, of course.
Obviously we have
\[
\Gamma_{e_1}=\{ M \in \Gamma; \, \alpha = 1, b=0, \gamma=0\}.
\]
Due to the form in Proposition \ref{prop:gsi} we already saw that
\[
\Gamma_{e_1} = \gsi \cap \Gamma.
\]
Now let $c$ be an arbitrary $0$-dimensional cusp for $\Gamma$.
Then there exists an $N \in O^{+}(S_1;\Q)$ with $N c = e_1$. A $\Gamma$-invariant function can be defined by
\[
\sum_{M: \Gamma_{c}\backslash \Gamma} 1|_k(NM),
\]
as long as this function converges absolutely.
If $M$ runs through a system of representatives in $\Gamma_{c}\backslash \Gamma$ then $NMN^{-1}$ runs through a system of $(N\Gamma N^{-1})_{e_1} \backslash (N \Gamma N^{-1})$.
Hence, for this function,
\begin{align}\label{eq:Eisenstein_cusp_def}
\sum_{M: \Gamma_{c}\backslash \Gamma} 1|_k(NM) = \sum_{M:(N\Gamma N^{-1})_{\infty} \backslash (N \Gamma N^{-1})} 1|_k(MN)
\end{align}
is well-defined.
The series is the same for two $\Gamma$-equivalent $0$-dimensional cusps.

We make this more explicit by considering $\Gamma= \GammaD$. Due to \cite{Schaps_Diss}, we can choose representatives of the cusps in the form
\begin{align}\label{eq:standardform_cusp}
c=\begin{pmatrix}
1 \\ 0 \\ \mathfrak{c} \\ 0 \\ Q(\mathfrak{c})
\end{pmatrix}, \, \mathfrak{c}\in L\dual, Q(\mathfrak{c}) \in \Z.
\end{align}
Also we can choose
$\widetilde{T}^{-1}_c$
for $N$, such that $\widetilde{T}_c e_1 = c$.

Thus, we end up with the following definition:
\begin{definition}
Let $c$ be a $0$-dimensional cusp given in the form of (\ref{eq:standardform_cusp}).
For $k>n+2$ we define the Eisenstein series for the orthogonal group $\GammaD$ attached to the cusp $c$ as
\[
E_{k,S,c}:= \frac{1}{2} \sum_{M:\widetilde{\Gamma_S}_{,c}\backslash \widetilde{\Gamma_S}} 1|_k(\widetilde{T}^{-1}_c M).
\]
\end{definition}
The series converges absolutely if and only if $k>n+2$, cf. \cite{Schaps_Diss}.

Every $0$-dimensional cusp is $\GammaD$-equivalent to one in the form of (\ref{eq:standardform_cusp}). 
However, being more precise, in each cusp class in $\mathcal{C}^{0}(\widetilde{\Gamma}_S)$ the elements are $\GammaD$-equivalent only up to a sign.
First, noting that $\Gamma_c = \Gamma_{-c}$, we get
\[
 E_{k,S,-c}= \frac{1}{2} \sum_{M:\widetilde{\Gamma_S}_{,-c}\backslash \widetilde{\Gamma_S}} 1|_k((-\widetilde{T}^{-1}_c) M) = \frac{1}{2} \sum_{M:\widetilde{\Gamma_S}_{,c}\backslash \widetilde{\Gamma_S}} \left((-1)^k 1|_k(\widetilde{T}^{-1}_c M)\right)= (-1)^k E_{k,S,c}.
\]
This motivates why we defined the equivalence of cusps just up to sign.

From the way we defined the Eisenstein series, we directly read off the modularity of $E_{k,S,c}$. However, for computing and later purposes an explicit form is useful and can be derived.
\begin{theorem}\label{thm:Eisenstein_explicit}
For each cusp $c$ of the form (\ref{eq:standardform_cusp}) and $k>n+2$,
\[
E_{k,S,c}(z) = \frac{1}{2} \sum_{\begin{subarray}{c}g \in L_1+c \\ \gcd(S_1 g) = 1\\ Q_1(g)=0\end{subarray}} j(g,z)^{-k},
\]
where $j(g,z):=j(M^{-1},z)$, if $g$ is the first column of $M^{-1}$, i.e. \\$j(g,z)= -\gamma Q_0(z)+(d,z)_0 + \delta$, $g=\begin{pmatrix}\delta \\ d \\ \gamma \end{pmatrix}$.

Moreover $E_{k,S,c} \in \M_k(\widetilde{\Gamma_S}).$
\end{theorem}

\begin{proof}
Using the absolute convergence, we may reorder the series.

The term $1|_k(\widetilde{T}^{-1}_c M)$ is equal to $j(\widetilde{T}^{-1}_c M,z)^{-k}$. The automorphy factor only depends on the last row. So, if $h\T S_1$ is the last row of an $N\in O(S_1;\R)$,
then the automorphy factor reads $j(N,z) = (h,Z)_1=j(h,z)$, where \[
Z= \begin{pmatrix}
-Q_0(z) \\ z \\ 1
\end{pmatrix}
\]
is the standard representative of $z$ in the projective model.

The last row of $\widetilde{T}^{-1}_c$ is $c\T S_1$.
Hence,
$$1|_k(\widetilde{T}^{-1}_c M) = ((({c}\T S_1 M)\T S_1^{-1},Z)_1)^{-k}= ((S_1^{-1} M\T S_1 {c},Z)_1)^{-k}=((M^{-1} {c},Z)_1)^{-k}.$$

For the claimed formula, set $g:=M^{-1} {c}$. The vector $g$ satisfies $$Q_1(g)=Q_1(M g)=Q_1({c})=0$$
and $\gcd(S_1 g) = \gcd(M\T S_1 {c}) = \gcd(S_1 {c})=1$.
Since $M \in \GammaD$, it follows $M^{-1}c = \lambda + c$ for a suitable $\lambda \in L_1$.

Over all, we proved so far $$E_{k,S,c}(z) = \frac{1}{2} \sum_{M:\widetilde{\Gamma}_{S,c}\backslash \widetilde{\Gamma_S}} j(M^{-1}c,z)^{-k}.$$
The vectors $g=M^{-1} c$ appear in the claimed series. It remains to be proved that every summand in the assertion appears exactly once.

By definition of the cosets, the orbit is
\begin{align} \GammaD c = \bigcup_{M \in \GammaD} \{ M^{-1} c \} &= \bigcup_{N \in \widetilde{\Gamma}_{S,c}} \bigcup_{M:\widetilde{\Gamma}_{S,c}\backslash \widetilde{\Gamma_S}}  \{M^{-1} N^{-1} c\} 
= \bigcup_{M:\widetilde{\Gamma}_{S,c}\backslash \widetilde{\Gamma_S}}  \{M^{-1} c\},\label{eq:union_orbit} \end{align}
since $N\in \widetilde{\Gamma}_{S,c}$.
If $M_1^{-1} c = M_2^{-1} c$, it follows that $(M_2 M^{-1})c =c$. As a consequence, $M_2 M^{-1} \in \widetilde{\Gamma}_{S,c}$ and $\widetilde{\Gamma}_{S,c} M_1 = \widetilde{\Gamma}_{S,c} M_2$.
Thus, the union on the right in (\ref{eq:union_orbit}) is disjoint.
Representatives of the orbits are given by (\ref{eq:standardform_cusp}), for example.
\end{proof}

We simply denote $E_{k,S} := E_{k,S,e_1}$ for the standard cusp. A slightly different Eisenstein series at the cusp $c$ is
\begin{align}\label{eq:Eisenstein_tilde_cusp}
E_{k,S,c}^{*}(z) := \frac{1}{2\zeta(k)}\sum_{\begin{subarray}{c}g\in L_1 + c\backslash \{0\} \\ Q_1(g)=0\end{subarray}} j(g,z)^{-k}.
\end{align}
Whenever $L$ is maximal, there is only one class of zero-dimensional cusps. From \cite{KS22}, it follows that
\[
E_{k,S}^{*}(z) = \frac{1}{2\zeta(k)}\sum_{\begin{subarray}{c}g\in L_1\dual\backslash \{0\} \\ Q_1(g)=0\end{subarray}} j(g,z)^{-k} = \frac{1}{2}\sum_{\begin{subarray}{c}g\in L_1\dual \\ \gcd(S_1 g)=1\\ Q_1(g)=0\end{subarray}} j(g,z)^{-k}= E_{k,S},
\]
since $j(m g,z)=m \cdot j(g,z)$ for all $m \in \Z$.
Hence the different Eisenstein series coincide.
Note that the Eisenstein series $E_{k,S,c}$ and $E_{k,S,c}^{*}$ do not coincide in general, as long as the lattice $L$ is not maximal.

Given $g\in L_1+c$ with $Q_1(g)=0$, we have $g/\gcd(S_1 g) \in L_1\dual$. Hence, $g/\gcd(S_1 g)$ is equivalent to a $0$-dimensional cusp $c^{*}$ which does not necessarily equal $c$, on the other hand.
Let $h\in L_1+c^{*}$, $N_{c^{*}}:=\operatorname{ord}_{L\dual/L}(c^{*})$.
If there exists a $\beta$ such that $\beta h\in L_1+c$, then
$$ (\beta + N_{c^{*}}) h = \beta h + N_{c^{*}} h \in L_1+c.$$
Furthermore $-h \in L_1-c^{*}$ then admitts $-\beta (-h)\in L_1+c$.
As a consequence, all such $\beta$ form cosets $\pm t_{c^{*}} \mod N_{c^{*}}$ for a suitable $t_{c^{*}}$. We get
\begin{align*}
 2\zeta(k) E_{k,S,c}^{*} &= \sum_{\alpha = 1}^\infty \sum_{\begin{subarray}{c}\alpha h \in L_1 +c\\
                                                                  Q_1(h)=0\\ \gcd(S_1 h) = 1
                                                                  \end{subarray}} j(\alpha h,z)^{-k}\\
 &=\sum_{c^{*}:\Ccal^0(\GammaD)} \sum_{\alpha \equiv \pm t_{c^{*}} \mod N_{c^{*}}}  \sum_{\begin{subarray}{c} h \in L_1 + c^{*}\\
                                                                  Q_1(h)=0\\ \gcd(S_1 h) = 1
                                                                  \end{subarray}} \alpha^{-k} j(h,z)^{-k}.
\end{align*}
Thus, the general relation between both Eisenstein series is
\begin{align}\label{eq:DiffEis}
E_{k,S,c}^{*} = \sum_{c^{*}:\Ccal^0(\GammaD)} \left( \frac{1}{\zeta(k)}\sum_{\alpha \equiv \pm t_{c^{*}} \mod N_{c^{*}}}\alpha^{-k} \right) E_{k,S,c^{*}},
\end{align}
where the inner (Dirichlet) series is set to $0$, if $t_{c^{*}}$ does not exist.

\section{Calculating the Fourier Coefficients of the singular term}
We decompose the vectors $g\in L_1\dual$ appearing in the Eisenstein series as above:
$$ g = \begin{pmatrix}
        \delta \\ d \\ \gamma
       \end{pmatrix}.$$

Split $E_{k,S,c}$ into two parts. Let $\EkSingC$ denote the singular term which is the contribution from the summands with $\gamma = 0$ in Theorem \ref{thm:Eisenstein_explicit}.
The second part $\EkCuspC$ stands for the summands where $\gamma \neq 0$ and we will call this the cusp term.
The first observation is that both parts, singular and cusp part of the series, admit a Fourier expansion, since
\[
T_\lambda g = \begin{pmatrix}
1 & -\lambda\T S_0 & -Q_0(\lambda) \\
0 & I^{(n)} & \lambda \\ 0 & 0 & 1
\end{pmatrix} \cdot \begin{pmatrix}
* \\ * \\ \gamma
\end{pmatrix} = \begin{pmatrix}
* \\ * \\ \gamma
\end{pmatrix}.
\]
The cocyle relation of the factor of automorphy $j(\cdot,\cdot)$ implies
$$j(g,z)^{-k}|_k T_\lambda = j(T_\lambda,z)^{-k}\cdot j(g, T_\lambda\langle z \rangle)^{-k} = j(T_\lambda^{-1}g,z)= j(T_{-\lambda} g,z).$$

When $g$ runs through the complete set $\{ g \in L_1+c; \, Q_1(g)=0, \gcd(S_1 g) = 1\}$, then $T_\lambda^{-1} g$ does so, too, since $T_\lambda \in \GammaD$
for $\lambda \in \Z^{n+2}$. Consequently, we have 
\[ \EkSingC(z) = \EkSingC(T_\lambda\langle z \rangle) = \EkSingC(z+\lambda)\] and similarly $\EkCuspC(z) = \EkCuspC(z+\lambda)$ for all $\lambda\in \Z^{n+2}$.

Now we calculate its Fourier expansion. 
First of all, we consider the singular term
\begin{align*}
 {E_{k,S,c}^{*}}^{(0)}(z) &= \frac{1}{2\zeta(k)}\sum_{\begin{subarray}{c} g\in L_1+c \backslash\{0\} \\ Q_1(g)=0 \\ \gamma = 0 \end{subarray}} j(g, z)^{-k}\\
 &=\delta_{c=e_1} + \frac{1}{2\zeta(k)} \sum_{\delta\in \Z} \sum_{\begin{subarray}{c} d\in L_0+\tilde{c}\backslash\{0\} \\ Q_0(d)=0\end{subarray}} ((d,z)_0+\delta)^{-k},
\end{align*}
where $\delta_{c=e_1}$ is the \emph{Kronecker} function, i.e. the value is $1$ if and only if $c$ is equivalent to $e_1$, and $\tilde{c}$ is the onto $L_0\dual$ projected vector of $c$, i.e. the first and last component is left out. 

We apply the well-known identity $$\sum_{m \in \Z} (\tau+m)^{-k} = \frac{(-2\pi i)^k}{\Gamma(k)} \cdot \sum_{r \in \N} r^{k-1} e^{2\pi i r \tau}$$
and get
$$
{E_{k,S,c}^{*}}^{(0)}(z) = \delta_{c=e_1}- \frac{2k}{B_k}\sum_{\begin{subarray}{c}d \in L_0+ \tilde{c}\backslash\{0\} \\ Q_0(d)=0\\ d_1 \geq 0\end{subarray}} \, \, \sum_{r \in \N} r^{k-1} e^{2\pi i r (d,z)_0},
$$
if $k$ is even. This gives
\begin{theorem}
Let $k>n+2$ be even. The Fourier coefficients of the singular term are given by
\begin{align*}
a_{{E_{k,S,c}^{*}}^{(0)}}(\lambda) = \left\{ \begin{array}{ll}
                             \delta_{c=e_1}, &\lambda = 0, \\
                             -\frac{2k}{B_k}\cdot \sum\limits_{\begin{subarray}{c}d|\gcd(S_0 \lambda) \\ \frac{1}{d}\lambda \in L_0+ c\end{subarray}} d^{k-1}, & 0\neq \lambda\in L_0\dual \cap \partial\Pos_S, \\
                             0, &\text{otherwise}.
                            \end{array} \right.
\end{align*}
\end{theorem}

The relation (\ref{eq:DiffEis}) yields
\begin{corollary}
 \label{lem:EisSing}
Let $k > n+2$. The non-zero Fourier coefficients of $\EkSingC$ occur with exponents $\lambda\in \partial \Pos_S$. The constant term reads
$$a_{\EkSingC}(0) = \begin{cases} 1, & c\in L_1 \text{ and $k$ even,}\\ 0, & \text{otherwise}.\end{cases}$$
 If $0 \neq \lambda \in \partial L_0\dual \cap \Pos_S$, then set $\varepsilon:=\gcd(S_0 \lambda)$. Clearly $h:=\frac{1}{\varepsilon}\lambda \in L_0\dual$ and there exists exactly one cusp $c^{*}$ such that $h \in L_0\dual \pm c^{*}$. Let $N:= \operatorname{ord}_{L_0\dual/L_0}(h)$. If a $t$ exists such that $th \in L_0+c$, then
 $$ a_{\EkSingC}(\lambda) = \frac{(-2\pi i)^k}{2\Gamma(k)} \cdot \sum_{d|\varepsilon} d^{k-1} \sum_{\begin{subarray}{c}0\neq\alpha \in \Z\\ \alpha \frac{\varepsilon}{d} \equiv t \mod N\end{subarray}} \mu(|\alpha|)\alpha^{-k}.$$
 Otherwise the Fourier coefficient vanishes.
\end{corollary}
A direct proof can also be found in \cite{Schaps_Diss}.

\section{Calculating the Fourier Coefficients of the cusp term}
Now consider the cusp term. So let $\gamma \neq 0$.
The computations are all done for $E_{k,S,c}$. By dropping the primitivity condition on the vectors, the computations are valid for $E_{k,S,c}^{*}$, too.

Firstly, note that $\delta$ is uniquely determined by
\begin{align}\label{gl:delta}
\delta = - \frac{1}{\gamma} Q_0(d).
\end{align}
We notice that the set $\mathcal{T}:=\{ T_\lambda; \, \lambda \in \Z^{n+2} \}\subset \widetilde{\Gamma}_S$ forms a group. 

We make use of the fact
\begin{align*}
T_\lambda \cdot  \begin{pmatrix}\delta \\ d \\ \gamma\end{pmatrix} =  \begin{pmatrix}
1 & -\lambda\T S_0 & - Q_0(\lambda) \\ 0 & I & \lambda \\ 0 & 0 & 1
\end{pmatrix}\cdot  \begin{pmatrix}\delta \\ d \\ \gamma\end{pmatrix} = \begin{pmatrix}
* \\ d+\gamma \cdot \lambda \\ \gamma
\end{pmatrix}.
\end{align*}
Reordering the summands gives
\begin{align*}
\EkCuspC(z) &= \frac{1}{2} \sum_{\begin{subarray}{c} g\in L_1+c \\ Q_1(g)=0\\ \gcd(S_1 g) = 1 \\ \gamma \neq 0 \end{subarray}}  j(g,z)^{-k} \\
&=
\frac{1}{2}\sum_{\gamma \neq 0} \; \sum_{\begin{subarray}{c} g = (\delta, d,\gamma)\in L_1+c \\ Q_1(g)= 0\\ \gcd(S_1g)=1 \\ d : L_0\dual / \gamma L_0\dual \end{subarray}}\;
\sum_{\lambda \in \Z^{n+2}} j(T_\lambda g, z)^{-k}.
\end{align*}
As a consequence of the cocyle relation we get
\begin{align*}
j(T_\lambda\cdot g, z ) = j(g, T_{-\lambda}\langle z \rangle) \cdot j(T_{-\lambda}, z)
= j(g , z - \lambda)
\end{align*}
and therefore
\begin{align}\label{eq:Cusp1}
\EkCuspC(z) &= \frac{1}{2}\sum_{\gamma \neq 0}\; \sum_{\begin{subarray}{c} g = (\delta, d,\gamma)\in L_1+c \\ Q_1(g)= 0\\ \gcd(S_1g)=1 \\ d : L_0\dual / \gamma L_0\dual \end{subarray}}\;
\sum_{\lambda \in \Z^{n+2}} j(g, z+\lambda )^{-k}.
\end{align}

The next observation is that we can write $j(g,z)= -\gamma \cdot Q_0\left(z-\frac{1}{\gamma}d\right)$ if $\gamma \neq 0$.
To prove this claim we put $\lambda = -\frac{1}{\gamma} d$ and calculate
\[
 T_\lambda\cdot g = \begin{pmatrix}
1 & \frac{1}{\gamma}d\T S_0 & -\frac{Q_0(d)}{\gamma^2} \\
0 & I^{(n)} & -\frac{1}{\gamma}d \\ 0 & 0 & 1
\end{pmatrix}\cdot \begin{pmatrix} \delta \\ d\\ \gamma\end{pmatrix} = \begin{pmatrix}
0 \\ 0 \\ \gamma
\end{pmatrix}.
\]
Because of (\ref{gl:delta}) we know that $\delta = - \frac{1}{\gamma}Q_0(d)$. That is why the first component of $T_\lambda g$ vanishes.
Consequently,
\[
-\gamma \cdot Q_0(z) =j(T_\lambda g, z) = j(g, z - \lambda) =j\left( g, z + \frac{1}{\gamma}d\right).
\]
We conclude from (\ref{eq:Cusp1})
\begin{align}\label{eq:Cusp2}
 \EkCuspC (z) = \frac{1}{2} \sum_{\gamma \neq 0} \; \sum_{\begin{subarray}{c} g = (\delta, d,\gamma)\in L_1+c  \\ Q_1(g)= 0\\ \gcd(S_1g)=1 \\ d : L_0\dual / \gamma L_0\dual \end{subarray}}\;\sum_{\lambda\in \Z^{n+2}}
\gamma^{-k} \left(-Q_0\left(z-\frac{1}{\gamma}d+\lambda \right)\right)^{-k}.
\end{align}

% For the Eisenstein series attached to our standard cusp $e_1$ we can assume that $k$ is even. 
% The contribution of the summands where $\gamma >0$ is the same as where $\gamma <0$. This leads to
% \begin{align}\label{eq:Cusp3}
%  \EkCusp (z) = \sum_{\gamma = 1}^\infty \; \sum_{\begin{subarray}{c} g = (\delta, d,\gamma) \in L_1 \\ S_1[g]= 0\\ \gcd(S_1g)=1 \\ d : \Z^{n+2} / \gamma\Z^{n+2} \end{subarray}}\;\sum_{\lambda\in \Z^{n+2}}
% \gamma^{-k} \left(-Q_0\left(z-\frac{1}{\gamma}d+\lambda \right)\right)^{-k}.
% \end{align}

The inner series is periodic. Due to \cite{Ibukiyama} an analogue of \cite[Lemma V.1.7]{K3} is
\begin{lemma}\label{lem:Fourier_quadratic_sum}
 Given $k>n+1$ one has
 \[
  \sum_{\lambda \in \Z^{n+2}} \left(-Q_0(z+\lambda)\right)^{-k} = \ConstCusp \cdot \sum_{\begin{subarray}{c}\lambda \in L_0\dual \cap \Pos_S\end{subarray}}
 Q_0(\lambda)^{k-\frac{n+2}{2}} e^{2\pi i (\lambda,z)_0},
 \]
 where $\ConstCusp := \frac{(2\pi)^{2k-\frac{n}{2}}}{\sqrt{\det(S)}\cdot \Gamma(k)\cdot \Gamma(k-\frac{n}{2})}$.
\end{lemma}

This gives the Fourier expansion.
\begin{theorem}\label{thm:Fourier1}
Let $k>n+2$ be even. The cusp term of the Eisenstein series has the Fourier expansion
\[
\EkCuspC(z) = \sum_{\begin{subarray}{c}\lambda \in L_0\dual \cap \Pos_{S}\end{subarray}} a_{\EkCusp}(\lambda) \cdot e^{2\pi i (\lambda, z)_0}
\]
with the Fourier coefficients
\[
a_{\EkCuspC}(\lambda) = \frac{\ConstCusp}{2} \cdot Q_0(\lambda)^{k-\frac{n+2}{2}} \sum_{\gamma \neq 0} \gamma^{-k}
\; \sum_{\begin{subarray}{c} g = (\delta, d,\gamma) \in L_1- c \\ Q_1(g)= 0\\ \gcd(S_1g)=1 \\ d : L_0\dual / \gamma L_0 \end{subarray}}  \exp\left(\frac{2\pi i}{\gamma} (\lambda,d)_0 \right).
\]
\end{theorem}

From here on, let $c=e_1$ be the standard cusp.
We interpret the series appearing in the Fourier coefficients as the value of the following \emph{Dirichlet series} at $s=k$:
\[
 L(s,a(\lambda,\cdot)) = \sum_{\gamma=1}^\infty a(\lambda, \gamma) \gamma^{-s},
\]
where the Dirichlet coefficient equals
\begin{align}\label{eq:DirichletCusp}
 a(\lambda, \gamma):=  \sum_{\begin{subarray}{c} g = (\delta, d,\gamma) \in L_1 \\ Q_1(g)= 0\\ \gcd(S_1g)=1 \\ d : L_0\dual / \gamma L_0 \end{subarray}}  
 \exp\left(\frac{2\pi i}{\gamma} (\lambda,d)_0 \right)
\end{align}

Before we explicitly compute $a(\lambda,\gamma)$, we need some preparations.
\begin{lemma}\label{lem:Dirichlet_coeff_prop}
\begin{enumerate}[(i)]
\item The Dirichlet coefficients are a \emph{multiplicative arithmetic function}, i.e. $a(\lambda,\gamma\gamma')=a(\lambda,\gamma)\cdot a(\lambda,\gamma')$ for coprime pairs $(\gamma,\gamma')$.
\item For all $\gamma \in \N$ and $K \in {O}(S_0;\Z)$ it holds $a(K \lambda, \gamma)=a(\lambda,\gamma)$.
\item Let $p\in \P$. Then there exists a $K \in O^{+}(S_0;\Z)$ such that
\[
K \lambda = p^{\nu_p(\gcd(S_0\lambda))}\begin{pmatrix}
l^{*} \\ \mu^{*} \\ m^{*}
\end{pmatrix},
\]
where $\nu_p(\gcd)$ is the $p$-adic valuation and $p \nmid m^{*}$.
\item Let $\beta | m$ and $\gcd(\beta,\gamma)=1$.
\[
 a\left(\begin{pmatrix}
         l \\ \mu \\ m
        \end{pmatrix},\gamma \right) = a\left(\begin{pmatrix}
                        \beta l \\ \mu \\ \frac{m}{\beta} 
                       \end{pmatrix},\gamma\right) \text{ and } a(\lambda,\gamma) = a\left(\frac{1}{\beta}\lambda,\gamma\right).
\]
\end{enumerate}
\end{lemma}
\begin{proof}

\begin{enumerate}[(i)]
\item Let $(\gamma,\gamma')$ be a pair of coprime numbers.
\begin{align*}
a(\lambda,\gamma)\cdot a(\lambda,\gamma') &=\sum_{\begin{subarray}{c} g = (*, u,\gamma) \in L_1 \\ Q_1(g)= 0\\ \gcd(S_1g)=1 \\ u : L_0 / \gamma L_0 \end{subarray}}
                                           \, \,
                                          \sum_{\begin{subarray}{c} g' = (*, v,\gamma') \in L_1 \\ Q_1(g')= 0\\ \gcd(S_1g')=1 \\ v : L_0 / \gamma' L_0 \end{subarray}}
 e_{\gamma\gamma'}((\lambda,\gamma' u+\gamma v)_0).
\end{align*}
Since $(\gamma,\gamma')$ are coprime, the map
\[
\Z^{n+2}/\gamma\Z^{n+2} \times \Z^{n+2}/\gamma'\Z^{n+2} \to \Z^{n+2}/(\gamma\gamma')\Z^{n+2}, (u,v)\mapsto \gamma'u+\gamma v
\]
is an isomorphism.
Furthermore,
\begin{align*}
Q_0(\gamma' u+\gamma v)= \gamma'^2 Q_0(u)+ \gamma\gamma' (u,v) + \gamma^2 Q_0(v).
\end{align*}
If $\gamma\gamma' | Q_0(\gamma' u+\gamma v)$, then it follows $\gamma|Q_0(u)$ and $\gamma'|Q_0(v)$, since $\gcd(\gamma,\gamma')=1$.
On the other hand, if $\gamma|Q_0(u)$ and $\gamma'|Q_0(v)$, we clearly have $\gamma\gamma'|Q_0(\gamma'u+\gamma v)$.

Hence, there exists a unique $\delta^{*}\in \Z$ such that
\begin{align*}
 g^{*} = \begin{pmatrix}
 \delta^{*}  \\ \gamma' u + \gamma v \\ \gamma \gamma'
 \end{pmatrix}
\end{align*}
admits $Q_1(g^{*})=0$.

Since $\gcd(S_1 g)=1=\gcd(S_1g')$ and $(\gamma,\gamma')$ is a coprime pair, by construction the property $\gcd(S_1 g^{*})=1$ is fulfilled.
Now the multiplicity follows.

\item Let $K \in O(S_0;\Z)$. Consider
\[
a(K\lambda,\gamma) = \sum_{\begin{subarray}{c} g = (\delta, d,\gamma) \in L_1 \\ Q_1(g)= 0\\ \gcd(S_1g)=1 \\ d : \Z^{n+2} / \gamma\Z^{n+2} \end{subarray}}  e_{\gamma}((K\lambda, d)_0).
\]
By replacing $g$ by $g':=R_K g$ a permutation of the summands is performed because $R_K\in O(S_1;\Z)\subset \GL_{n+4}(\Z)$ and thus $Q_1(g) = Q_1(R_K g)$.

\item Without loss of generality we may assume that $\nu_p(\gcd(S_0 \lambda))=0$.
Let $\lambda=\begin{pmatrix}
l \\ \mu \\ m
\end{pmatrix}$.
Suppose that $p|l$ and $p|m$. Consider
\[
V K_u V \lambda = V \begin{pmatrix}
                   1 & u\T S & Q(u) \\ 0 & I & u \\ 0 & 0 & 1
                  \end{pmatrix}\cdot \begin{pmatrix}
m \\ \mu \\ l
\end{pmatrix} = \begin{pmatrix}
m+(u,\mu)+Q(u)l \\ \mu+lu \\ l
\end{pmatrix} V.
\]
One has $p|m$ and $p|Q(u)l$. However we have $p\nmid S\mu$ since $\nu_p(\gcd(S_0 \lambda)=0$. Thus we find a $u\in \Z^n$ such that $p\nmid u\T S\mu=(u,\mu)$ and it follows that $p\nmid (K_u V\lambda)_1$.
The multiplication with $V\in O(S_0;\Z)$ from the left interchanges the first and last component. This is the assertion.
\item To keep the index of the sums readable, we assume the following decomposition of the vectors
\[
 g = \begin{pmatrix} \delta \\ d \\ \gamma\end{pmatrix}\in \Z^{n+4}, \, \,  d = \begin{pmatrix} d_1 \\ \mathfrak{d} \\ d_2 \end{pmatrix}\in \Z^{n+2} \text{ and }\lambda^{*} = \begin{pmatrix}
                        \beta l \\ \mu \\ \frac{m}{\beta} 
                       \end{pmatrix}.
\]
The vector $\lambda$ is decomposed as before.

\begin{align*}
 a\left(\lambda^{*},\gamma\right) &= \sum_{\begin{subarray}{c} Q_1(g)= 0\\ \gcd(S_1g)=1 \\ d : \Z^{n+2} / \gamma\Z^{n+2} \end{subarray}} e_{\gamma}\left( d_1 \cdot \frac{m}{\beta} - (\mathfrak{d},\mu) + d_2 \beta l\right). 
\end{align*}
Now we claim that replacing $g$ by $g'$ performs a permutation of the summands, where we put
\[
 g' = \begin{pmatrix} \delta^{*} \\ \beta^{-1} d_1 \\ \mathfrak{d} \\ \beta d_2 \\ \gamma\end{pmatrix}
\]
and $\delta^{*}$ is uniquely determined. Here $\cdot^{-1}$ denotes a representative of an inverse in $\Z/\gamma\Z$.

We know that $\delta = -\frac{1}{\gamma}(d_1 d_2 - Q(\mathfrak{d}))$. There exists an $r\in \Z$ such that $\beta \beta^{-1} = 1 + \gamma r$ as $\beta^{-1}\in \Z$ is a representative of the inverse of $\beta$.
We have 
\begin{align*}
\delta^{*} &= -\frac{1}{\gamma} ( d_1 d_2 \beta \beta^{-1} - Q(\mathfrak{d}))= -\frac{1}{\gamma} (d_1 d_2 (1 + r \gamma ) - Q(\mathfrak{d}))= \delta - d_1d_2 r.
\end{align*}
We conclude $\gcd(S_1 g') = \gcd(\gamma, \beta^{-1} d_1, S\mathfrak{d}, \beta d_2, \delta^{*})=\gcd(S_1 g)$.
Since both $d_1+ \gamma\Z \mapsto \beta^{-1} d_1 + \gamma\Z$ and $d_2+ \gamma\Z \mapsto \beta d_2 + \gamma\Z$ are isomorphisms our claim is true.

For the second part of our assertion we use the same idea.
\begin{align*}
 a\left(\frac{1}{\beta}\lambda,\gamma\right) &= \sum_{\begin{subarray}{c} Q_1(g)= 0\\ \gcd(S_1g)=1 \\ d : \Z^{n+2} / \gamma\Z^{n+2} \end{subarray}} e_{\gamma}\left(\left(d,\frac{1}{\beta}\lambda\right)_0\right).
\end{align*}
Here we replace $g$ by
\[
 g'= \begin{pmatrix} \beta^2 \cdot \delta \\ \beta d \\ \gamma\end{pmatrix}.
\]
Because of $\gcd(\beta,\gamma)=1$ a reordering of the summands is performed and we have
\begin{align*}
 a\left(\frac{1}{\beta}\lambda,\gamma\right) &= \sum_{\begin{subarray}{c} Q_1(g)= 0\\ \gcd(S_1g)=1 \\ d : \Z^{n+2} / \gamma\Z^{n+2} \end{subarray}} e_{\gamma}\left(\left(\beta d,\frac{1}{\beta}\lambda\right)_0\right)
 = \sum_{\begin{subarray}{c} Q_1(g)= 0\\ \gcd(S_1g)=1 \\ d : \Z^{n+2} / \gamma\Z^{n+2} \end{subarray}} e_{\gamma}\left(\left(d,\lambda\right)_0\right) = a(\lambda,\gamma).
\end{align*}
\end{enumerate}
\end{proof}

The arithmetic multiplicity of the Dirichlet coefficients implies an \emph{Euler product expansion}
\[
L(s,a(\lambda,\cdot)) =\prod_{p \in \P} \sum_{\nu = 0}^\infty a(\lambda,p^{\nu}) p^{-\nu s}.
\]

Let $\alpha,\beta,\gamma \in \N$. The \emph{Kloosterman sum} is given by
\[
K(\alpha, \beta; \gamma ) := \sum_{\begin{subarray}{c}j=1 \\ \gcd(j,\gamma)=1\end{subarray}}^\gamma e_\gamma(\alpha j + \beta j^{-1}).
\]

The convoluted Dirichlet coefficient can be expressed in terms of Kloosterman sums.
\begin{proposition}\label{prop:conv_form}
We define the arithmetic function $b(\lambda,\cdot)$ as the Dirichlet convolution
\[
b(\lambda,\cdot) = 1*a(\lambda,\cdot).
\]
Let the localization $L_p$ be maximal and let $\lambda \in L_0\dual \cap \Pos_S$. Put $\nu_\lambda:=\nu_p(\gcd(S_0 \lambda))$. Then one has
\[
b(\lambda,p^\nu) = \sum_{t=0}^{\min(\nu,\nu_\lambda)} \sum_{\begin{subarray}{c} \mathfrak{l}\in \Z^n/p^\nu\Z^n \\ Q(\mathfrak{l}) \equiv 0 \mod p^{t}
                                                \end{subarray} } \left\{
 e_{p^{\nu}}\left(-\left(\mathfrak{l},\mu \right)\right) \cdot K\left(l,\frac{m Q(\mathfrak{l})}{p^{2t}};p^{\nu-t}\right)\right\}.
\]
\end{proposition}

\begin{proof}
For the first step we prove the formula
\begin{align}\label{eq:conv_form}
  b(\lambda,\gamma)=\sum_{\begin{subarray}{c}
                                            v \in \Z^{n+2}/\gamma\Z^{n+2} \\
                                            Q_0(v) \equiv 0 \mod \gamma
                                           \end{subarray}} e_\gamma((\lambda, v)_0).
\end{align}
It is sufficient to consider only prime powers $\gamma=p^\nu$ by virtue of Lemma \ref{lem:Dirichlet_coeff_prop}(i). The Dirichlet convolution is certainly multiplicative, too.

One has
\begin{align*}
b(\lambda, p^\nu) &= (1*a(\lambda, \cdot)(p^\nu) = \sum_{j=0}^\nu a(\lambda,p^j) \\
&= \sum_{j=0}^\nu \sum_{\begin{subarray}{c}g=(*,d\T,p^j)\T \\ Q_1(g)=0 \\ \gcd(S_1 g)=1 \\ d:\Z^{n+2}/p^j\Z^{n+2} \end{subarray}} e_{p^\nu}\left( p^{\nu-j}(d,\lambda)_0\right).
\end{align*}

Let $g=(\delta,d,p^\nu)$ with $Q_0(d)\equiv 0 \mod p^\nu$. Choose $\delta$ such that $\gamma \delta = Q_0(d)$. In the ring $\Z/p^\nu \Z$, $\delta$ is uniquely determined. Note that $Q_0(d)$ is well-defined modulo $p^\nu$, since
\[
Q_0(d+p^\nu v) = Q_0(d)+ p^\nu (d,v)_0 + p^{2\nu} Q_0(v)
\]
for every $v \in \Z^{n+2}$.

If $\gcd(S_1 g)=\gcd(g)$ holds, then $\gcd\left( S_1 g^{*}\right)=1$, where $g^{*}=\frac{1}{\gcd(g)}g\in L_1$ is well-defined.
As a consequence of $\gamma = p^\nu$, we already know that $\gcd(g)=p^j$ for some $j\geq 0$ and the summand is
\[
e_{p^\nu}((d,\lambda)_0) = e_{p^\nu}(p^j(d^{*},\lambda)_0) = e_{p^{\nu-j}}((d^{*},\lambda),
\]
where $g^{*}$ a unique vector with $\gcd(S_1 g^{*})=1$.

The maximality of $L_p$ ensures that $\gcd(S_1 g)= \gcd(g)$ for all $Q_1(g)=0$ and $\gamma=p^\nu$. Otherwise, this would be a contradiction to Proposition \ref{prop:local_appl}.

Obviously, we have $b(\lambda,1)=1$. So let $\nu \geq 1$.

Due to Lemma \ref{lem:Dirichlet_coeff_prop} (ii) and (iii) we may additionally assume that $p\nmid p^{-\nu_\lambda}m$. Recall that $\nu_\lambda=\gcd(S_0 \lambda)$. The convoluted coefficient reads
\begin{align*}
  b(\lambda,p^\nu) &= \sum_{\begin{subarray}{c}d \mod p^\nu \\
                                            d_1 d_2 \equiv Q(\mathfrak{d}) \mod p^\nu
                                           \end{subarray}} e_{p^\nu}( d_1 m - (\mathfrak{d},\mu) + d_2 l).
%&= \sum_{\begin{subarray}{c}d \mod p^\nu \\
%                                            d_1 d_2 \equiv Q(\mathfrak{d}) \mod p^\nu
%                                           \end{subarray}} e_{p^{\nu-\nu_\lambda}}( d_1 m - (\mathfrak{d},\mu) + d_2 l).
 \end{align*}
 
Consider the homomorphism
 \[
  h_\alpha: \Z/p^\nu\Z \to \Z/p^\nu\Z, \beta \mapsto \alpha\beta
 \]
 with its kernel
 \begin{align*}
  \ker(h_\alpha) = \{ p^{\nu-t}+p^\nu\Z,2p^{\nu-t}+p^\nu\Z,\cdots,p^{\nu}\Z\},
 \end{align*}
 where $p^t=\gcd(\alpha,p^\nu)$.
 There are $p^t$ elements in the kernel and the image is given by
 \begin{align*}
  \im(h_\alpha) = \{ \alpha+p^\nu\Z, 2\alpha+p^\nu\Z,\cdots,p^{\nu-t}\alpha+p^\nu\Z \}
  = \{ p^t+p^\nu\Z, 2p^t+p^\nu\Z,\cdots,p^\nu\Z \}.
 \end{align*}
 When $Q(\mathfrak{l})\equiv 0 \mod p^t$, there exist $p^t$ solutions of
 \[
  \alpha \beta \equiv Q(\mathfrak{l}) \mod p^{\nu}
 \]
 for $\beta$. If $\beta^{*}$ is such a solution, the other solutions are given by $\beta^{*}+jp^{\nu-t}$,$j=0,...,p^{t}-1$.
 
 Otherwise, if $Q(\mathfrak{l})\not\equiv 0 \mod p^{t}$, then $Q(\mathfrak{l})+p^\nu\Z$ is not part of the image of $h_\alpha$.
 
 Now, let us reorder the finite sum in the following way: We sort $\alpha=d_2\in\Z/p^\nu\Z$ by its greatest common divisor with $p^\nu$ and put $\alpha = p^t \alpha^{*}$ with $p \nmid \alpha^{*}$. This gives
 \begin{align*}
  b(\lambda,&p^\nu) = \sum_{t=0}^\nu \sum_{\begin{subarray}{c}
                                           \alpha^{*}=1 \\ p \nmid \alpha^{*}
                                          \end{subarray}}^{p^{\nu-t}}
                                          \sum_{\begin{subarray}{c} \mathfrak{l}\in \Z^n/p^\nu\Z^n \\ Q(\mathfrak{l}) \equiv 0 \mod p^{t}
                                                \end{subarray} }\sum_{j=0}^{p^{t-1}} e_{p^{\nu}} (m(\beta^{*} +jp^{v-t})-(\mathrm{l}, \mu) +\alpha^{*} p^t l) \\
&=  \sum_{t=0}^\nu \sum_{\begin{subarray}{c}
                                           \alpha^{*}=1 \\ p \nmid \alpha^{*}
                                          \end{subarray}}^{p^{\nu-t}}
                                          \sum_{\begin{subarray}{c} \mathfrak{l}\in \Z^n/p^\nu\Z^n \\ Q(\mathfrak{l}) \equiv 0 \mod p^{t}
                                                \end{subarray} }e_{p^{\nu}} (\beta^{*} m-(\mathfrak{l}, \mu) +\alpha^{*} p^t l) \cdot \left(\sum_{j=0}^{p^t-1} \exp\left(\frac{2\pi i m}{p^t}j\right)\right).
 \end{align*}
We recognize a geometric sum which is given by
 \begin{align}\label{eq:geometric_sum}
  \sum_{j=0}^{p^t-1} \exp\left(\frac{2\pi i m}{p^t}j\right)= \begin{cases}p^t, & t \leq \nu_\lambda, \\
  0, & t > \nu_\lambda. \end{cases}
 \end{align}
 
 If $p^t|Q(\mathfrak{l})$, then the congruence is equivalent to
 \begin{align*}
 \alpha^{*} \beta^{*} \equiv \frac{Q(\mathfrak{l})}{p^t} \mod p^{\nu-t}.
 \end{align*}
 Due to $p\nmid \alpha^{*}$ we can invert $\alpha^{*}$ in $\Z/p^{\nu-t}\Z$ and obtain
 \[
 \beta^{*} \equiv {\alpha^{*}}^{-1} \frac{Q(\mathfrak{l})}{p^t} \mod p^{\nu-t}.
 \]

Whenever $t\leq \min(\nu,\nu_\lambda)$, the inner sum reads
 \begin{align*}
 &\sum_{\begin{subarray}{c}\alpha^{*}=1 \\ p \nmid \alpha^{*}
                                          \end{subarray}}^{p^{\nu-t}}
                                          \sum_{\begin{subarray}{c} \mathfrak{l}\in \Z^n/p^\nu\Z^n \\ Q(\mathfrak{l}) \equiv 0 \mod p^{t}
                                                \end{subarray} }
 e_{p^\nu}(-(\mathfrak{l},\mu)) \cdot e_{p^{\nu-t}} (\beta^{*} m p^{-t} +\alpha^{*} l)\\
 =& \sum_{\begin{subarray}{c}\alpha^{*}=1 \\ p \nmid \alpha^{*}
                                          \end{subarray}}^{p^{\nu-t}}
                                          \sum_{\begin{subarray}{c} \mathfrak{l}\in \Z^n/p^\nu\Z^n \\ Q(\mathfrak{l}) \equiv 0 \mod p^{t}
                                                \end{subarray} }
e_{p^\nu}(-(\mathfrak{l},\mu)) \cdot e_{p^{\nu-t}}\left({\alpha^{*}}^{-1} \cdot \frac{Q(\mathfrak{l})}{p^t} \cdot m p^{-t} + \alpha^{*} l \right).
 \end{align*}
 Recognizing the Kloosterman sum in the formula completes the proof.
\end{proof}

\begin{remark}\label{rm:Bruinier}
\begin{enumerate}[(i)]
\item Let $L$ be maximal and $k>n+2$ be even. Let $\lambda \in L_0\dual\cap \Pos_S$ and $\gcd(S_0 \lambda)=1$. Then one has
\[
b(\lambda,p^\nu) = \sum_{\begin{subarray}{c} \mathfrak{l}\in \Z^n/p^\nu\Z^n \end{subarray} } \left\{
 e_{p^{\nu}}\left(\left(\mathfrak{l},\mu \right)\right) \cdot K\left(l,Q(\mathfrak{l})m,p^{\nu}\right)\right\}.
\]
\item It follows from (i) that the Fourier-Jacobi coefficient of index $1$ is a Jacobi-Eisenstein series and it is also linked to vector valued Eisenstein series examined in \cite{Bruinier_Eisenstein}. For details, refer to \cite{Schaps_Diss}.
\item Considering $E_{k,S,c}^{*}$, for even $k>n+2$, the formula in Theorem \ref{thm:Fourier1} for the Fourier coefficients is very similar
\begin{align}\label{eq:tilde_Fourier}
a_{{\EkCuspC}^{*}}(\lambda) = \frac{\ConstCusp}{2\zeta(k)} \cdot Q_0(\lambda)^{k-\frac{n+2}{2}} \sum_{\gamma \in \N} \gamma^{-k}\sum_{\begin{subarray}{c} g = (\delta, d,\gamma) \in L_1\pm c \\ Q_1(g)= 0 \\ d : L_0\dual / \gamma L_0 \end{subarray}}  \exp\left(\frac{2\pi i}{\gamma} (\lambda,d)_0 \right).
\end{align}
The Dirichlet series in (\ref{eq:tilde_Fourier}) equals $L(s,b(\lambda,\cdot))$. Hence, Proposition \ref{prop:conv_form} holds for $E_{k,S}^{*}$, too.
Moreover, the assumption on $L_p$ to be maximal can be dropped. The assumption in the proof of Proposition \ref{prop:conv_form} was necessary to get rid of the primitivity condition on the vectors.
In the following results, the assumption on the maximality of $L_p$ is only a consequence of Proposition \ref{prop:conv_form}.
\end{enumerate}
\end{remark}

In the case of primitive vectors, the Fourier expansion has already been considered several times. We extend the results in \cite{Bruinier_Eisenstein} and \cite{Mocanu_PhD} for the non-primitive cases.
The procedure is pretty much the same. The calculations become slightly longer.

\begin{theorem}\label{thm:Fourier2}
Consider the product of the Dirichlet series
\[
\zeta(s-n)\cdot L(s,b(\lambda,\cdot)) = \prod_{p \in \P} \sum_{\nu = 0}^\infty N^{*}(\lambda,p^\nu) p^{\nu(1- s)}.
\]
Let $p \in \P$ and the localization $L_p$ be maximal. 
Let $\ell_\lambda:= \min\{ n \in \N;\, n \cdot \lambda \in L_0\}$ denote the level of $\lambda \in L_0\dual$ and $r := \nu_p(\ell_\lambda)$ its $p$-adic valuation. As before we set $\nu_\lambda:= \nu_p(\gcd(S_0 \lambda))$.

The representation number $N^{*}(\lambda,p^\nu)$ is a sum
\[
N^{*}(\lambda,p^\nu) = \sum_{t=0}^{\min(\nu,\nu_\lambda)} N(\lambda,p^\nu,p^t).
\]

The representation numbers satisfy the following properties:
\begin{enumerate}[(i)]
\item For any $\nu \geq \nu_\lambda\geq t$ the representation number is given by
\begin{align*}
N(\lambda,p^\nu,p^t) = \#\left\{ v \in \Z^n/p^{\nu+r}\Z^n; \begin{array}{l} \, \,\,\, v\, \, \, \, \equiv -p^{r+t-\nu_\lambda} \mu \,\,\,\,\,\,\,\,\,\, \mod p^r \\ Q(v) \equiv - Q_0(p^{r-(\nu_\lambda-t)} \lambda) \mod p^{\nu-(\nu_\lambda-t)+2r}\end{array} \right\}.
\end{align*}
\item For any $t \leq \nu < \nu_\lambda$
\[
N(\lambda,p^\nu,p^t) = \#\{ v \in \Z^n/p^\nu\Z^n; \, Q(v) \equiv 0 \mod p^t\}.
\]
\item For any $\nu \geq 0$ we have
\[
N(\lambda,p^\nu,p^{t^{*}}) = \#\left\{ v \in \Z^n/p^{\nu+r}\Z^n; \begin{array}{l} \, \,\,\, v\, \, \, \, \equiv -p^{r} \mu \,\,\,\,\,\,\,\,\,\, \mod p^r \\ Q(v) \equiv - Q_0(p^{r} \lambda) \mod p^{\nu+2r}\end{array} \right\}
\]
for $t^{*} = \min(\nu,\nu_\lambda)$.
\item Put $w_p = 1+2\nu_p(2\ell_\lambda Q_0(\lambda))$. Then the equality
 \[
  N^{*}(\lambda,p^{\nu +1}) = p^{n-1} N^{*}(\lambda,p^\nu)
 \]
 holds for any $\nu \geq w_p$.
\end{enumerate}
\end{theorem}
\begin{proof}
We continue our calculations where we ended up with in Proposition \ref{prop:conv_form}.

First, we consider the inner sums again:
\begin{align*}
&\sum_{\begin{subarray}{c} \mathfrak{l}\in \Z^n/p^\nu\Z^n \\ Q(\mathfrak{l}) \equiv 0 \mod p^{t}
                                                \end{subarray} } \left\{
 e_{p^{\nu}}\left(-\left(\mathfrak{l},\mu \right)\right) \cdot K\left(l,\frac{m Q(\mathfrak{l})}{p^{2t}};p^{\nu-t}\right)\right\}\\ =& \sum_{\begin{subarray}{c}\alpha^{*}=1 \\ p \nmid \alpha^{*}
                                          \end{subarray}}^{p^{\nu-t}}
                                          \sum_{\begin{subarray}{c} \mathfrak{l}\in \Z^n/p^\nu\Z^n \\ Q(\mathfrak{l}) \equiv 0 \mod p^{t}
                                                \end{subarray} }
 e_{p^{\nu-t}}\left({\alpha^{*}}^{-1} \cdot \frac{m Q(\mathfrak{l})}{p^{2t}} -p^{-t}(\mathfrak{l},\mu) + \alpha^{*} l \right).
\end{align*}

By replacing $\mathfrak{l}$ by $\alpha^{*} \mathfrak{l}$ a permutation of the summands in the inner sum is performed, since $p\nmid \alpha^{*}$.
We get
\begin{align}
&\sum_{\begin{subarray}{c}\alpha^{*}=1 \\ p \nmid \alpha^{*}
                                          \end{subarray}}^{p^{\nu-t}}
                                          \sum_{\begin{subarray}{c} \mathfrak{l}\in \Z^n/p^\nu\Z^n \\ Q(\mathfrak{l}) \equiv 0 \mod p^{t}
                                                \end{subarray} }
 e_{p^{\nu-t}}\left({\alpha^{*}}^{-1} \cdot \frac{m Q(\alpha^{*} \mathfrak{l})}{p^{2t}} - \left(\alpha^{*}\mathfrak{l}, \mu p^{-t}\right) +\alpha^{*} l\right)\nonumber \\
 =& \sum_{\begin{subarray}{c} \mathfrak{l}\in \Z^n/p^\nu\Z^n \\ Q(\mathfrak{l}) \equiv 0 \mod p^{t}
                                                \end{subarray} }
                                                \sum_{\begin{subarray}{c}\alpha^{*}=1 \\ p \nmid \alpha^{*}
                                          \end{subarray}}^{p^{\nu-t}}
 e_{p^{\nu-t}}\left(\alpha^{*} \cdot\left( \frac{m Q(\mathfrak{l})}{p^{2t}} -\left(\mathfrak{l},\mu p^{-t}\right)+l \right) \right).\label{eq:Gauss_sum1}
\end{align}
The inner sum is a Ramanujan sum. In this case, evaluate the sum by splitting it into two geometric sums:
\begin{align*}
&\sum_{\begin{subarray}{c}\alpha^{*}=1\end{subarray}}^{p^{\nu-t}}
 e_{p^{\nu-t}}\left(\alpha^{*}\left(\frac{m Q(\mathfrak{l})}{p^{2t}}
 -\left(\mathfrak{l},\mu p^{-t}\right)+l \right) \right) \\
 -&\sum_{\begin{subarray}{c}\alpha^{*}=1\end{subarray}}^{p^{\nu-t-1}}
 e_{p^{\nu-t}}\left(\alpha^{*}p \left( \frac{m Q(\mathfrak{l})}{p^{2t}} -\left(\mathfrak{l},\mu p^{-t}\right)+l \right) \right).
\end{align*}

Evaluating similarly to (\ref{eq:geometric_sum}) we obtain that (\ref{eq:Gauss_sum1}) becomes
\begin{align*}
&p^{\nu-t} \cdot \#\left\{ \mathfrak{l} \in \Z^n/p^\nu\Z^n; \,\begin{array}{l} Q(\mathfrak{l})\equiv 0 \mod p^t, \\
\frac{m Q(\mathfrak{l})}{p^t}-(\mathfrak{l},\mu)+p^{t} l \equiv 0 \mod p^\nu\end{array} \right\}\\
-& p^{\nu-t-1} \cdot \#\left\{ \mathfrak{l} \in \Z^n/p^\nu\Z^n; \,\begin{array}{l} Q(\mathfrak{l})\equiv 0 \mod p^t, \\ \frac{m Q(\mathfrak{l})}{p^t}-(\mathfrak{l},\mu)+p^{t} l \equiv 0 \mod p^{\nu-1}\end{array} \right\},
\end{align*}
if $t<\nu$. Using the notation of 
\[
N(\lambda,p^\nu,p^t)=\#\left\{ \mathfrak{l} \in \Z^n/p^\nu\Z^n; \begin{array}{l} Q(\mathfrak{l})\equiv 0 \mod p^t,\\ \frac{m Q(\mathfrak{l})}{p^t}-(\mathfrak{l},\mu)+p^{t} l \equiv 0 \mod p^\nu\end{array} \right\},
\]
(\ref{eq:Gauss_sum1}) equals
\[
p^{\nu-t}N(\lambda,p^\nu,p^t) - p^{\nu-t-1} p^n N(\lambda,p^{\nu-1},p^t).
\]
Note that the representation number $N(\lambda,p^\nu,p^t)$ can be expressed as 
\begin{align}\label{eq:rep_final}
N(\lambda,p^\nu,p^t) = \left\{v\in \Z^n/p^\nu\Z^n; \, \begin{array}{l} Q(v)\equiv 0 \mod p^t,\\ p^{\nu_\lambda-t}Q(v)-(v,\mu)+p^{t-\nu_\lambda} lm \equiv 0 \mod p^\nu\end{array}\right\},
\end{align}
since $mp^{-\nu_\lambda}$ is invertible in $\Z/p^\nu\Z$.

In summary, we have computed
\begin{align}
b(\lambda,p^\nu) &= \sum_{t=0}^{\min(\nu,\nu_\lambda)} p^t p^{\nu-t} N(\lambda,p^\nu,p^t) - \sum_{t=0}^{\min(\nu-1,\nu_\lambda)} p^t p^{\nu-t-1} p^n N(\lambda,p^{\nu-1},p^t)\nonumber \\
&= \sum_{t=0}^{\min(\nu,\nu_\lambda)} p^{\nu} N(\lambda,p^\nu,p^t) - \sum_{t=0}^{\min(\nu-1,\nu_\lambda)} p^{\nu-1} p^n N(\lambda,p^{\nu-1},p^t)\label{eq:Gauss_sum2}
\end{align}

Let us compute the coefficients for $L(s,c(\lambda,\cdot)) := \zeta(s-n)\cdot L(s,b(\lambda,\cdot))$. The \emph{Dirichlet convolution} leads to
\begin{align*}
\zeta(s-n) \cdot L(s,b(\lambda,\cdot)) = \prod_{p \in \P} \sum_{\nu=0}^\infty c(\lambda,p^\nu) p^{-\nu s},
\end{align*}
where
\[
c(\lambda,p^\nu) = \sum_{\delta | p^\nu} \left(\frac{p^\nu}{\delta}\right)^n \cdot a(\lambda,\delta) = \sum_{\nu_\delta=0}^\nu p^{(\nu-\nu_\delta)n} \cdot a(\lambda,p^{\nu_\delta}).
\]
Plugging in the formula in (\ref{eq:Gauss_sum2}) yields
\begin{align*}
c&(\lambda,p^\nu) = \sum_{\nu_\delta=0}^\nu p^{(\nu-\nu_\delta)n} \left(\sum_{t=0}^{\min(\nu_\delta,\nu_\lambda)} p^{\nu_\delta} N(\lambda,p^{\nu_\delta},p^t) - \sum_{t=0}^{\min(\nu_\delta-1,\nu_\lambda)} p^{\nu_\delta-1} p^n N(\lambda,p^{\nu_\delta-1},p^t)\right) \\
&= \sum_{\nu_\delta=0}^\nu  \sum_{t=0}^{\min(\nu_\delta,\nu_\lambda)} p^{(\nu-\nu_\delta)n + \nu_\delta} N(\lambda,p^{\nu_\delta},p^t) - \sum_{\nu_\delta=0}^{\nu} \sum_{t=0}^{\min(\nu_\delta-1,\nu_\lambda)} p^{(\nu-\nu_\delta)n+\nu_\delta-1+n} N(\lambda,p^{\nu_\delta-1},p^t).
\end{align*}
Recognizing a telescoping sum, we obtain
\[
c(\lambda,p^\nu) = \sum_{t=0}^{\min(\nu,\nu_\lambda)} p^{\nu} N(\lambda,p^{\nu},p^t) = p^\nu N^{*}(\lambda,p^\nu).
\]
The formula for $N(\lambda,p^\nu,p^t)$ can be simplified in every case.

\begin{enumerate}[(i)] 
\item Let $\nu \geq \nu_\lambda$.
The term
\[
p^{\nu_\lambda-t}Q(v)-(v,\mu)+p^{-\nu_\lambda+t} lm
\]
is always divisible by $p^{\nu_\lambda-t}$. Hence, $N(\lambda,p^\nu,p^t)$ equals
\[
\#\left\{ v \in \Z^n/p^\nu\Z^n; \begin{array}{l} Q(v)\equiv 0 \mod p^t,\\ Q(v) -(v,p^{t-\nu_\lambda} \mu)+p^{2(t-\nu_\lambda)} lm \equiv 0 \mod p^{\nu-(\nu_\lambda-t)}\end{array} \right\}.
\]
Here we can drop the first condition, since $p^t|(x,p^{t-\nu_\lambda} \mu)$ and $p^t|p^{2(t-\nu_\lambda)}lm$.
If the second congruence is satisfied, this already implies that $p^t|Q(v)$.
Consequently, $N(\lambda,p^\nu,p^t)$ is equal to
\begin{align}
\label{eq:rep_form1}
\#\left\{ v \in \Z^n/p^\nu\Z^n; \, Q(v) -(v,p^{t-\nu_\lambda} \mu)+p^{2(t-\nu_\lambda)} lm \equiv 0 \mod p^{\nu-(\nu_\lambda-t)} \right\}.
\end{align}
Put $\Delta := Q_0(p^{-\nu_\lambda}\lambda)$. A computation yields
\begin{align}
N(\lambda,p^\nu,p^t)&=\#\left\{ \mathfrak{l} \in \Z^n/p^\nu\Z^n; p^{2r} Q(\mathfrak{l}-p^{t-\nu_\lambda} \mu) + p^{2r+2t} \Delta \equiv 0 \mod p^{\nu-(\nu_\lambda-t)+2r} \right\}\nonumber \\
&=\#\left\{ \mathfrak{l} \in \Z^n/p^\nu\Z^n; Q(p^r \mathfrak{l}- p^{t-\nu_\lambda} p^r\mu) + p^{2r+2t} \Delta \equiv 0 \mod p^{\nu-(\nu_\lambda-t)+2r} \right\} \nonumber \\
&=\#\left\{ v \in \Z^n/p^{\nu+r}\Z^n; \begin{array}{l} \, \,\,\, v\, \, \, \, \equiv p^rp^{t-\nu_\lambda} \mu \,\,\,\,\,\,\,\,\,\, \mod p^r \\ Q(v) \equiv - p^{2r+2t} \Delta \mod p^{\nu-(\nu_\lambda-t)+2r}\end{array} \right\}.\label{eq:Bruinier_form} 
\end{align}

\item Whenever $\nu <\nu_\lambda$, the congruence
\[
p^{\nu_\lambda-t}Q(v)-(v,\mu)+p^{-\nu_\lambda+t} lm \equiv 0 \mod p^\nu
\]
is trivial and can be dropped due to $p^t|Q(v)$.

\item This is just a combination of (i) and (ii). Note that $p^{\nu+2r} | Q_0(p^r\lambda)$ if $\nu < \nu_\lambda$.

\item Use Lemma 5 in \cite{Bruinier_Eisenstein} with $\widetilde{w_p}=1+2\nu_p(2\ell_\lambda p^{2t}\Delta)-(\nu_\lambda-t)$.
\end{enumerate}
\end{proof}

Define the \emph{local Euler factor}
\[
L_{\lambda}(s,p) =  (1-p^{n-s}) \sum_{\nu=0}^{w_p-1} N^{*}(\lambda,p^\nu) p^{\nu(1-s)} + N^{*}(\lambda,p^{w_p})p^{w_p(1-s)}
\]
and we simply have
\[
L(s,b(\lambda,\cdot)) = \prod_{p \in \P} L_{\lambda}(s,p).
\]
In case of a non-maximal lattice $L$ we can follow the same procedure and get a product into local Euler factors
\[
L(s,b(\lambda,\cdot)) = \prod_{\begin{subarray}{c}p \in \P\\ L_p \text{ maximal}\end{subarray}} L_{\lambda}(s,p) \cdot \prod_{\begin{subarray}{c}p \in \P\\ L_p \text{ not maximal}\end{subarray}} L_{\lambda}(s,p).
\]
The local Euler factors where $L_p$ is not maximal are just finitely many ones. Following the method of \cite{Bruinier_Eisenstein}, we obtain
\begin{theorem}\label{thm:Fourier3}
Let $\lambda\in L_0\dual\cap\Pos_S$ and $k>n+2$ even.
Let $\chi_\mathcal{D}(m) = \left(\frac{\mathcal{D}}{m}\right)$.
%\begin{itemize}
If $n$ is even, put $\mathcal{D}=(-1)^{n/2}\cdot \det(S)$. Then the Fourier coefficient equals
\begin{align*}
a_{E_{k,S}}(\lambda) = \frac{\ConstCusp}{\zeta(k)\cdot L(k-\frac{n}{2},\chi_\mathcal{D})} \cdot Q_0(\lambda)^{k-\frac{n+2}{2}} \cdot \prod_{p|2 \ell_\lambda Q_0(\lambda) \det(S)} \frac{L_{\lambda}(k,p)}{1-\chi_\mathcal{D}(p)p^{\frac{n}{2}-k} }.
\end{align*}
If $n$ is odd, write $Q_0(\lambda) = n_0 f^2$, where $n_0\in \Q$ and $f\in \N$, such that $(f,2\det(S))=1$ and $\nu_{q}(n_0)\in\{0,1\}$ for all primes $q$ with $(q,2\det(S))=1$.
 Let $n_1:=n_0 \ell_\lambda^2$ and $\mathcal{D}=2(-1)^{(n+1)/2}n_1\det(S)$. Then the Fourier coefficient equals
 \begin{align*}
 a_{E_{k,S}}(\lambda) = \frac{\ConstCusp \cdot L(k-\frac{n+1}{2},\chi_\mathcal{D})}{\zeta(k)\cdot \zeta(2k-(n+1))} \cdot Q_0(\lambda)^{k-\frac{n+2}{2}} \cdot \prod_{p|2 \ell_\lambda^2 Q_0(\lambda)\det(S)} \frac{1-\chi_\mathcal{D}(p)p^{\frac{n+1}{2}-k}}{1-p^{n+1-2k}}L_\lambda(k,p).
 \end{align*}
\end{theorem}

\begin{corollary}\label{cor:rational}
Let $L$ be maximal. The Fourier coefficients $a_{E_{k,S}}(\lambda)$ are rational numbers.
\end{corollary}

\begin{remark}
 \begin{enumerate}[(i)]
  \item Let $L$ be maximal, $n$ even, $k > n+2$ even and $\lambda \in L_0\dual \cap \Pos_S$ primitive, i.e. $\gcd(S_0 \lambda) = 1$. 
  Proposition 2.19 in \cite{Mocanu_PhD} yields
\begin{align*}
a_{E_{k,S}}(\lambda) =& (-1)^{\lceil \frac{n}{4}\rceil} \cdot \frac{2k}{B_k} \cdot \frac{2(k-\frac{n}{2})}{B_{k-\frac{n}{2},\chi_\mathfrak{f}}} \cdot \frac{|\mathfrak{f}|^{k-\frac{n+2}{2}}}{f_0} \cdot \chi_\mathfrak{f}(\mathfrak{d}) \cdot \left(\frac{Q_0(\lambda)}{\mathfrak{d}}\right)^{k-\frac{n+2}{2}} \sigma_{k-\frac{n+2}{2},\chi_\mathfrak{f}}(\mathfrak{d}) \\
& \cdot \prod_{\begin{subarray}{c} p|\det(S) \\ p \nmid \mathfrak{f}\end{subarray}} \frac{1}{1-\chi_\mathfrak{f}(p)p^{\frac{n}{2}-k}} \cdot \prod_{p | 2 \det(S)} \frac{L_\lambda(k,p)}{1-\chi_\mathcal{D}(p)p^{\frac{n}{2}-k}},
\end{align*}
where $\mathcal{D}= (-1)^{n/2}\det(S)$, $\mathfrak{f}$ is the fundamental discriminant, $\mathfrak{D}=f_0^2 \cdot \mathfrak{f}$ and
\[
\mathfrak{d} = \prod_{\begin{subarray}{c}p | \ell_\lambda^2Q_0(\lambda) \\ p \nmid 2\det(S)\end{subarray}} p^{\nu_p(\ell_\lambda^2Q_0(\lambda))}.
\]
\item Let $L$ be unimodular of rank $n$ and let $k>n+2$ be even. For $\gcd(\lambda)=1$, Woitalla \cite{Woitalla_Fourier} calculated
\[
a_{E^{\bullet}_{k,S}}(\lambda) = \frac{2k}{B_k}\cdot \frac{(2k-n)}{B_{k-n/2}}\cdot \sigma_{k-\frac{n}{2}-1}(Q_0(\lambda)).
\]
\item If $S=(2)$, the Eisenstein series equal the classical Siegel-Eisenstein series.
The Eisenstein series for $O(2,4)$, if $L$ is maximal even, correspond to Siegel-type Hermitian Eisenstein series of degree $2$ over the imaginary quadratic field $\Q(\sqrt{-m})$, where $m$ is squarefree, cf. \cite{We}.
Considering the $D_4$-lattice, we obtain Siegel-Eisenstein series of degree $2$ over the Hamiltonian quaternions, cf. \cite{K3}.
Calculations show that the formulas above yield the same Fourier coefficients, at least up to the finitely many local Euler factors at the "bad" places.
 \end{enumerate}
\end{remark}

%%%%%%%%%%%%%%%%%%%%%%%%%%%%%%%%%%%%%%%%%%%%%%%%%%%%%%%%%%%%%%%
%%%%%%%%%%%%%%%%%%%%%%%%%%%%%%%%%%%%%%%%%%%%%%%%%%%%%%%%%%%%%%%

\section{Eisenstein series and the Maaß space}

\begin{theorem}\label{thm:Eis_Maass}
Let $p\in \P$ and the localization $L_p$ be maximal. Then the local Maaß condition is satisfied.
\end{theorem}
\begin{proof}
The singular part $\EkSing$ and the cusp term $\EkCusp$ can be treated separately.

We prove the local Maaß condition for both terms.

At first, we prove the Maaß condition for the singular term.
Let $0\neq\lambda \in L_0\dual \cap \partial \Pos_S$ and let $\varepsilon=\gcd(S_0 \lambda)= p^r \cdot \varepsilon'$, where $p \nmid \varepsilon'$. Recall the representation of the Fourier coefficient in Corollary \ref{lem:EisSing}.
We put $h:=\frac{1}{\varepsilon}\lambda \in L_0\dual$, so there exists exactly one cusp $c^{*}$ such that $h \in L_0\dual \pm c^{*}$. Let $N:= \operatorname{ord}_{L_0\dual/L_0}(h)$. 
 The Fourier coefficients satisfy
$a_f(\lambda) = a_f(K \lambda)$ for all $K \in O^{+}(S_0;\Z)$. In virtue of Lemma \ref{lem:Dirichlet_coeff_prop}(iii) without loss of generality we may assume that $\lambda = (l,\mu,p^r m)$, where $p \nmid m$.
To prove the local Maaß condition, we need to verify
\begin{align}\label{eq:EisSingMaass2}
a_{\EkSing}(\lambda) = a_{\EkSing}(\lambda^{*}) + p^{k-1} a_{\EkSing}\left(\frac{1}{p} \lambda\right), \, \lambda^{*}=(p^r l, \mu, m).
\end{align}
The left hand side of the equation is
\begin{align*}
a_{\EkSing}(\lambda)&=\frac{(-2\pi i)^k}{\Gamma(k)} \cdot\sum_{d|\varepsilon'}\sum_{j=0}^r d^{k-1}p^{j(k-1)} \sum_{\begin{subarray}{c}\alpha \in \N\\ \alpha \frac{\varepsilon'}{d}p^{r-j} \equiv 0 \mod N\end{subarray}} \mu(\alpha)\alpha^{-k}.
\end{align*}
Now let the localization $L_p$ be maximal. Then $p \nmid N= \operatorname{ord}_{L_0\dual/L_0}(h)$ due to Proposition \ref{prop:local_appl}.
In particular, the congruence condition $\frac{\varepsilon'}{d}p^{r-j} \equiv 0 \mod N$ is equivalent to $\frac{\varepsilon'}{d} \equiv 0 \mod N$, since $p$ is invertible in $\Z/N\Z$.
Using
\begin{align*}
a_{\EkSing}\left(\frac{1}{p}\lambda\right)&= p^{1-k} \cdot \frac{(-2\pi i)^k}{\Gamma(k)}  \cdot\sum_{d|\varepsilon'}\sum_{j=1}^r d^{k-1}p^{j(k-1)} \sum_{\begin{subarray}{c}\alpha \in \N\\ \alpha \frac{\varepsilon'}{d}p^{r-j} \equiv 0 \mod N\end{subarray}} \mu(\alpha)\alpha^{-k}.
\end{align*}
and noting that $\lambda^{*}$ satisfies $\gcd(S_0 \lambda^{*})= \varepsilon'$, one recognizes (\ref{eq:EisSingMaass2}).

Next, consider the cusp term.
Let $\lambda = (l,\mu,p^rm) \in L_0\dual \cap \Pos_S$ and $p \in \P$, $r \in \N$ and $p\nmid m$.

Again, we prove the local Maaß condition.
Recall that
\[
a_{\EkCusp}(\lambda) = \frac{\ConstCusp}{\zeta(k)} \cdot Q_0(\lambda)^{k-\frac{n+2}{2}} \cdot L(k,b(\lambda,\cdot)).
\]
We split
\[
L(k,b(\lambda,\cdot)) = \prod_{\begin{subarray}{c}q \in \P \\ p \neq q\end{subarray}}
\left(\sum_{\nu \geq 0}  N^{*}(\lambda,q^\nu) q^{\nu(1-k)}\right) \cdot \left(\sum_{\nu \geq 0}  N^{*}(\lambda,p^\nu) p^{\nu(1-k)}\right).
\]
By virtue of Lemma \ref{lem:Dirichlet_coeff_prop}(iv)
\[
a(\lambda,q^\nu) = a(\lambda^{*},q^\nu)=a\left(\frac{1}{p}\lambda,q^\nu\right), \, q\neq p.
\]
This property leads to $b(\lambda,q^{\nu}) = (1*a)(\lambda,q^{\nu})$ for all $q\neq p$ and $\nu \geq 0$.

Deal with the fixed prime $p$. Here, we have
\[
N^{*}(\lambda^{*},p^\nu) = \#\left\{ v \in \Z^n/p^{\nu+r}\Z^n; \begin{array}{l} \, \,\,\, v\, \, \, \, \equiv -p^{r+ \nu_\lambda} \mu \,\,\,\,\,\,\,\,\,\, \mod p^r \\ Q(v) \equiv - Q_0(p^{r} \lambda^{*}) \mod p^{\nu+2r}\end{array} \right\}
\]
as a result of $\nu_p(S_0 \lambda^{*}) = 0$ and Theorem \ref{thm:Fourier2}.
Comparing this formula with Theorem \ref{thm:Fourier2} (iii) leads to
\[
N(\lambda,p^\nu,p^{\min(\nu_\lambda,\nu)}) = N^{*}(\lambda^{*},p^\nu).
\]
Now let $0\leq t < \min(\nu,\nu_\lambda)$. To avoid the study of two cases we here consider the formula (\ref{eq:rep_final}) in the proof of Theorem \ref{thm:Fourier2}. Let $\nu_\lambda:=\nu_p(\gcd(S_0 \lambda))$ as already used before.
\[
N(\lambda,p^\nu,p^t)=\#\left\{ \mathfrak{l} \in \Z^n/p^\nu\Z^n; \begin{array}{l} Q(\mathfrak{l})\equiv 0 \mod p^t,\\ \frac{m Q(\mathfrak{l})}{p^t}-(\mathfrak{l},\mu)+p^{t} l \equiv 0 \mod p^\nu\end{array} \right\}.
\]
In the second congruence the left side is always divisible by $p$ and the representation number reads
\[
N(\lambda,p^\nu,p^t)=\#\left\{ \mathfrak{l} \in \Z^n/p^\nu\Z^n; \begin{array}{l} Q(\mathfrak{l})\equiv 0 \mod p^t,\\ \frac{m Q(\mathfrak{l})}{p^{t+1}}-\frac{1}{p}(\mathfrak{l},\mu)+p^{t-1} l \equiv 0 \mod p^{\nu-1}\end{array} \right\}.
\]
If we compare it with
\[
N\left(\frac{1}{p}\lambda,p^{\nu-1},p^{t}\right)=\#\left\{ \mathfrak{l} \in \Z^n/p^{\nu-1}\Z^n; \begin{array}{l} Q(\mathfrak{l})\equiv 0 \mod p^t,\\ \frac{m Q(\mathfrak{l})}{p^{t+1}}-(\mathfrak{l},\frac{1}{p}\mu)+p^{t-1} l \equiv 0 \mod p^{\nu-1}\end{array} \right\},
\]
we see the same congruences.

If $\mathfrak{l} \mod p^{\nu-1}$ satisfies both of these congruences, then $\mathfrak{l} +p^{\nu-1}v$ satisfies the congruences above, where the vectors $v$ run through a system $\mod p$. Hence,
\[
N(\lambda,p^\nu,p^t) = p^n \cdot N\left(\frac{1}{p}\lambda,p^{\nu-1},p^t\right)
\]
and
\[
\sum_{t=0}^{\min(\nu,\nu_\lambda)-1} N(\lambda,p^\nu,p^t) = \sum_{t=0}^{\min(\nu,\nu_\lambda)-1} p^n N\left(\frac{1}{p}\lambda,p^{\nu-1},p^t\right)= p^nN^{*}\left(\frac{1}{p}\lambda,p^{\nu-1}\right).
\]
This yields the local Maaß condition for $p$.
\end{proof}

The combination of Theorem \ref{thm:Eis_Maass} and Proposition \ref{prop:local} immediately leads to
\begin{corollary}\label{cor:Eis_Maass_maximal}
Let $L$ be maximal. Then $E_{k,S} \in \M_k^{*}(\Gamma_S).$
\end{corollary}

Theorem \ref{thm:Eis_Maass} always holds at almost all places $p$. Using an analogue of \cite{HK} for $O(2,n+2)$ would give $E_{k,S} \in \M_k^{*}(\Gamma_S)$, even if $L$ is not maximal.

%============================================
%\printbibliography

\bibliographystyle{plain}
\bibliography{biblio}
\end{document}